\documentclass[10pt,twoside,a4paper]{amsart}
\usepackage{graphicx}
\usepackage{natbib}
\usepackage{amssymb}
\usepackage{amsmath}
\usepackage{amsthm}
\usepackage{fancyhdr}
\usepackage{latexsym}
\usepackage{tikz}
\usepackage{multirow}
\usepackage{centernot}
\usepackage[hidelinks]{hyperref}
\usetikzlibrary{arrows,decorations.pathmorphing,backgrounds,positioning,fit,arrows.meta} 
\usetikzlibrary{matrix,arrows,calc}
\DeclareMathAlphabet{\mathpzc}{OT1}{pzc}{m}{it}
\usepackage{fixme}
\fxsetup{targetlayout=color}

\usepackage{MnSymbol}

\DeclareRobustCommand{\unEdge}{%
	\mathrel{%
		\text{%
			\ooalign{$\rightfootline$\cr\reflectbox{$\rightfootline$}\cr}%
		}%
	}%
}

\newcommand{\bSigma}{\boldsymbol \Sigma}

\begin{document}

\newtheorem{theorem}{Theorem}[section]
\newtheorem{proposition}[theorem]{Proposition}
\newtheorem{lemma}[theorem]{Lemma}
\newtheorem{definition}[theorem]{Definition}
\newtheorem{corollary}[theorem]{Corollary}

\theoremstyle{definition}
\newtheorem{ex}[theorem]{Example}
\newtheorem{rem}[theorem]{Remark}

\title[A Trek Rule for the Lyapunov Equation]{A Trek Rule for the Lyapunov Equation}

\author[N. R. Hansen]{Niels Richard Hansen}

\address{Department of Mathematical Sciences, University
of Copenhagen Universitetsparken 5, 2100 Copenhagen \O, Denmark,
Niels.R.Hansen@math.ku.dk}

\keywords{Gaussian Markov processes, graphical models, Lyapunov equation, treks}
    
\subjclass{62R01, 62A09, 15A24}

\maketitle

\begin{abstract}

The Lyapunov equation is a linear matrix equation characterizing the cross-sectional steady-state covariance matrix of a Gaussian Markov process. We show a new version of the trek rule for this equation, which links the graphical structure of the drift of the process to the entries of the steady-state covariance matrix. In general, the trek rule is a power series expansion of the covariance matrix in the entries of the drift and volatility matrices. For acyclic models it simplifies to a polynomial in the off-diagonal entries of the drift matrix. Using the trek rule we can give relatively explicit formulas for the entries of the covariance matrix for some special cases of the drift matrix. These results illustrate notable differences between covariance models entailed by the Lyapunov equation and those entailed by linear additive noise models. To further explore differences and similarities between these two model classes, we use the trek rule to derive a new lower bound on the marginal variances in the acyclic case. This sheds light on the phenomenon, well known for the linear additive noise model, that the variances in the acyclic case tend to increase along a topological ordering of the variables.

\end{abstract}

\section{Introduction}

With $M$ any $d \times d$ matrix and $C$ a $d \times d$ 
positive semidefinite matrix, we can define a Gaussian Markov process $(X_t)_{t \geq 0}$ 
on $\mathbb{R}^d$ as a solution to the stochastic differential equation
\begin{equation} \label{eq:SDE}
\mathrm{d}X_t = MX_t \mathrm{d}t + C^{1/2} \mathrm{d}W_t,
\end{equation}
where $W_t$ is a $d$-dimensional Brownian motion, see, e.g.,  
\cite{Jacobsen:1993} or Section 3.7 in \cite{Pavliotis:2014}.
We call $M$ the 
\emph{drift matrix} and $C$ is called the \emph{diffusion matrix} or the 
\emph{volatility matrix}.

The Markov process has 
a steady-state distribution if and only if the matrix $M$ is stable, 
that is, if and only if all eigenvalues of $M$ have strictly negative
real parts. In this case, the steady-state distribution is Gaussian with mean
$0$ and covariance matrix $\Sigma$, which is the unique
solution to the (continuous) Lyapunov equation
\begin{equation} \label{eq:Lyapunov}
M\Sigma + \Sigma M^T + C = 0,
\end{equation}
see, e.g., \cite[Theorem 2.12]{Jacobsen:1993}.

When the process is stationary, that is, when it is started in its (unique)
Gaussian steady-state distribution $\mathcal{N}(0, \Sigma)$, with $\Sigma$
solving \eqref{eq:Lyapunov}, then the cross-sectional distribution of $X_t$ is
$\mathcal{N}(0, \Sigma)$ at any time $t$. Questions regarding estimation and
identification of $M$ and/or $C$ from cross-sectional observations of the
process were treated in \cite{varando2020} and \cite{Dettling:2023}.

The main contribution of this paper is Theorem \ref{thm:altrep}, which gives a
novel representation of $\Sigma$ in terms of the drift and volatility matrices
$M$ and $C$.  The non-zero entries 
 of $M$ and $C$ define a mixed graph, see Section \ref{sec:graph}, 
 which is used in Section \ref{sec:treks} to define a collection of treks between any two nodes. These are special walks in the graph, 
 and we give two versions of the trek rule in 
 Proposition \ref{prop:trek} and Theorem \ref{thm:altrep}, respectively, that 
 express
 the entries of the solution $\Sigma$ to \eqref{eq:Lyapunov} as a sum 
 of trek weights over all treks. 
 
 Trek rules are well known for linear additive noise models, also known as 
 linear structural equation models, see \cite{sullivant:2010} and
 Section 4 in \cite{drton2018}. A first version of the trek rule for the Lyapunov equation
 was presented by \cite{varando2020}, but our new version in Theorem \ref{thm:altrep}
 is more explicit and 
 easier to use an interpret.  A special case of our general trek rule, valid
 for certain acyclic models, was recently presented as Proposition 4.3 in
 \cite{boege2024}, where it was used to characterize the conditional
 independencies that hold in the steady-state distribution. We treat the 
 acyclic case in detail in Section \ref{sec:acyclic}, where we 
 use of the trek rule to derive a novel lower bound
 on the marginal variances. 

\subsection{The Lyapunov equation} The Lyapunov equation 
\eqref{eq:Lyapunov} is a linear matrix equation. Using the Kronecker 
product, we can rewrite it as   
\begin{equation} \label{eq:LyapunovKron}
(M \otimes I + I \otimes M)\mathrm{vec}(\Sigma) = -\mathrm{vec}(C),
\end{equation}
where $I$ is the $d \times d$ identity matrix and $\mathrm{vec}(A)$
denotes the vectorization of the matrix $A$. When $M$ is stable, 
the $d^2 \times d^2$ matrix $(M \otimes I + I \otimes M)$ is invertible, 
and the unique solution to \eqref{eq:LyapunovKron} is given by
\begin{equation} \label{eq:LyapunovSol}
\mathrm{vec}(\Sigma) = -(M \otimes I + I \otimes M)^{-1}\mathrm{vec}(C).
\end{equation}
The representation \eqref{eq:LyapunovSol}  shows 
that $\Sigma_{ij}$ is a rational function of the entries of $M$ and $C$, 
but \eqref{eq:LyapunovSol}  is not very explicit, nor is it efficient for numerical computation 
of the solution. There is a vast literature on numerical methods for 
solving the Lyapunov equation efficiently, see, e.g., \cite{Simoncini:2016} 
for a review. See also \cite{Gajic:1995} for a comprehensive treatment of the
Lyapunov equation and its applications.

We will not be particularly concerned with numerical 
methods but rather with giving a new representation of
the solution that is interpretable in terms of the graphical structure
of the drift matrix $M$. To this end, we will use the well known 
integral representation of the solution given by 
\begin{equation} \label{eq:intrep}
    \Sigma = \int_0^{\infty} e^{tM} C
    e^{tM^T} \mathrm{d} t.
\end{equation}
Here, $e^{tM}$ denotes the matrix exponential of $tM$, and the 
integral is understood as a matrix integral. It is convergent when
$M$ is stable. See \cite{varando2020} and the references therein for
further details on the integral representation of the solution.

\subsection{Graphs} \label{sec:graph}
We will represent the non-zero entries of the drift matrix
$M$ and the volatility matrix $C$ by a mixed graph $\mathcal{G}$ with nodes $[d]
= \{1, \ldots, d\}$, directed edges, $\rightarrow$, and blunt edges, $\unEdge$.
The blunt edges, introduced in \cite{varando2020} and \cite{Mogensen:2022}, will be
used to represent the covariance structure of the noise process. In the
literature on linear additive noise models, bidirected edges, rather than blunt
edges, are conventionally used to represent the covariance structure of the
noise. The primary reason for using a different notation is that a
bidirected edge is also used to represent a common latent factor, and in the
context of stochastic processes, a latent factor cannot in general be
captured by a correlated noise process. See \cite{Mogensen:2022} for further
details. 

\begin{definition} \label{dfn:graph}
    A pair of matrices $(M, C)$ is
    compatible with a mixed graph $\mathcal{G}$ if
    $m_{ji} \neq 0$ implies $i \rightarrow j$ and $c_{ij} \neq 0$ implies
    $i \unEdge j$.
\end{definition}

\begin{ex} \label{ex:graph0} We consider the specific model with $d = 5$, 
    and using $.$ to denote zero entries, 
  $$M = \left(\begin{array}{rrrrr}
  -1 & 0.5 & . & 0.2 & . \\ 
    -1 & -1 & 0.2 & . & . \\ 
    . & . & -1 & 0.5 & . \\ 
    . & . & . & -1 & 1 \\ 
    . & . & 1 & . & -1 \\ 
  \end{array}\right) \quad \text{and} \quad C = \left(\begin{array}{rrrrr}
    1 & . & . & . & . \\ 
    . & 1 & . & . & . \\ 
    . & . & 1 & . & . \\ 
    . & . & . & 1 & . \\ 
    . & . & . & . & 1 \\ 
  \end{array}\right).$$

This pair $(M, C)$ is compatible with the graph $\mathcal{G} = ([5],
E)$ as given by (A) in Figure \ref{fig:graph0}. The eigenvalues of $M$ are 
\[ 
    -0.206,  -1 \pm 0.707i, \text{ and } -1.397 \pm 0.687i
\]
  with all real parts strictly negative, and $M$ is thus stable. Solving the 
  Lyapunov equation gives the steady-state covariance matrix (rounded to 3 decimals)
  
\[
       \Sigma = \left(\begin{array}{rrrrr}
        0.496 & -0.091 & 0.123 & 0.207 & 0.151 \\
        -0.091 & 0.594 & 0.013 & -0.038 & -0.005 \\
        0.123 & 0.013 & 0.838 & 0.676 & 0.647 \\
        0.207 & -0.038 & 0.676 & 1.412 & 0.912 \\
        0.151 & -0.005 & 0.647 & 0.912 & 1.147 \\
    \end{array}\right).
\]  

\end{ex}

\begin{figure}
    \centering
    \begin{tikzpicture}
    \tikzset{vertex/.style = {shape=circle,draw,minimum size=1.5em, inner sep = 0pt, outer sep = 1pt}}
    \tikzset{vertexdot/.style = {shape=circle,fill=red,color=red,draw,minimum size=0.5em, inner sep = 0pt}}
    \tikzset{vertexhid/.style = {shape=rectangle,draw,minimum size=1.5em, inner sep = 0pt}}
    \tikzset{edge/.style = {->,> = latex', thick}}
    \tikzset{edgeun/.style = {-, thick}}
    \tikzset{edgebi/.style = {<->,> = latex', thick}}
    \tikzset{edgeblunt/.style = {{|[scale=0.5]}-{|[scale=0.5]}, thick}}

    \node at (0, 2.5) {(A)};
    \node at (5, 2.5) {(B)};
    
    \node[vertex] (A) at (-1.2, 1.2) {$1$};
    \node[vertex] (B) at (1.2, 1.2) {$2$};
    \node[vertex] (C) at (1.2, -1.2) {$3$};
    \node[vertex] (D) at (-1.2, -1.2) {$5$};
    \node[vertex] (E) at (0, 0) {$4$};
    
    \draw[edge] (A) edge[loop left, color=blue] node[left, color=black, scale=0.7] {$-1$} (A);
    \draw[edge] (A) edge[bend right  = 20, color=blue] node[above, color=black, scale=0.7] {$-1$} (B);
    \draw[edge] (B) edge[loop right, color=blue] node[right, color=black, scale=0.7] {$-1$} (B);
    \draw[edge] (B) edge[bend right  = 20, color=blue] node[above, color=black, scale=0.7] {$0.5$} (A);
    \draw[edge] (C) edge[loop right, color=blue] node[right, color=black, scale=0.7] {$-1$} (C);
    \draw[edge] (C) edge[bend right  = 15, color=blue] node[right, color=black, scale=0.7] {$0.2$} (B);
    \draw[edge] (C) edge[bend left  = 15, color=blue]  node[above, color=black, scale=0.7] {$1$} (D);
    \draw[edge] (E) edge[bend left = 15, color=blue]  node[left, pos=0.3, color=black, scale=0.7] {$0.2$ \ } (A);
    \draw[edge] (E) edge[bend left  = 15, color=blue]  node[right, pos=0.3, color=black, scale=0.7] {\ $0.5$} (C);
    \draw[edge] (E) edge[loop above, color=blue] node[right, pos=0.7, color=black, scale=0.7] {$-1$} (E);
    \draw[edge] (D) edge[bend left = 15, color=blue]   node[left, pos=0.7, color=black, scale=0.7] {$1$ \ } (E);
    \draw[edge] (D) edge[loop left, color=blue] node[left, color=black, scale=0.7] {$-1$} (D);
    
    \draw[edgeblunt, loop above, color=red] (A) to (A);
    \draw[edgeblunt, loop above, color=red] (B) to (B);
    \draw[edgeblunt, loop below, color=red] (C) to (C);
    \draw[edgeblunt, loop below, color=red] (D) to (D);
    \draw[edgeblunt, loop below, color=red] (E) to (E);
    
    \node[vertex] (A1) at (-1.2+5, 1.2) {$1$};
    \node[vertex] (B1) at (1.2+5, 1.2) {$2$};
    \node[vertex] (C1) at (1.2+5, -1.2) {$3$};
    \node[vertex] (E1) at (5, 0) {$4$};
    \node[vertex] (D1) at (-1.2+5, -1.2) {$5$};
    
    \draw[edge] (A1) edge[bend right  = 20, color=blue] node[above, color=black, scale=0.7] {$-1$} (B1);
    \draw[edge] (B1) edge[bend right  = 20, color=blue] node[above, color=black, scale=0.7] {$0.5$} (A1);
    \draw[edge] (C1) edge[bend right  = 15, color=blue] node[right, color=black, scale=0.7] {$0.2$} (B1);
    \draw[edge] (C1) edge[bend left  = 15, color=blue]  node[above, color=black, scale=0.7] {$1$} (D1);
    \draw[edge] (E1) edge[bend left = 15, color=blue]  node[left, pos=0.3, color=black, scale=0.7] {$0.2$ \ } (A1);
    \draw[edge] (E1) edge[bend left  = 15, color=blue]  node[right, pos=0.3, color=black, scale=0.7] {\ $0.5$} (C1);
    \draw[edge] (D1) edge[bend left = 15, color=blue]   node[left, pos=0.7, color=black, scale=0.7] {$1$ \ } (E1);

    \draw[edgeblunt, loop above, color=red] (A1) to (A1);
    \draw[edgeblunt, loop above, color=red] (B1) to (B1);
    \draw[edgeblunt, loop below, color=red] (C1) to (C1);
    \draw[edgeblunt, loop below, color=red] (D1) to (D1);
    \draw[edgeblunt, loop below, color=red] (E1) to (E1);
    
    \end{tikzpicture} 
    \caption{Mixed graph (A) representing a model with $d = 5$
      variables and with a diagonal $C$-matrix. 
      The edge labels of the directed (blue) edges 
      are the values of the non-zero entries in the $M$-matrix,
        while all the blunt (red) edges correspond to the 
        diagonal entries of $C$ all being $1$. 
        The mixed graph (B) is the same as (A) but with 
        all directed self-loops removed. We call (B) the base 
        graph of (A).
        \label{fig:graph0}}
\end{figure}
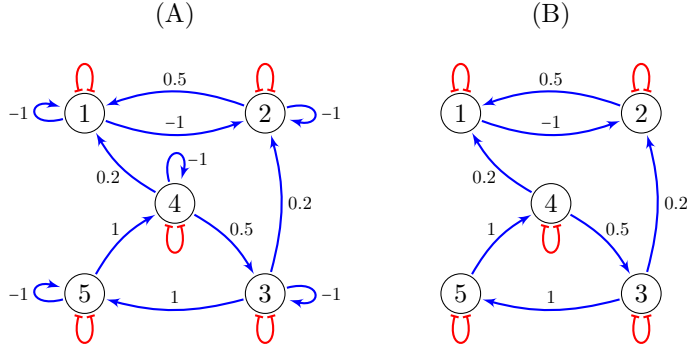
    
We use standard graph terminology, see, e.g., \cite{Maathuis:2019}. 
Specifically we let a walk be a sequence of (not necessarily unique) 
nodes where each node is connected
to the next by an edge. A walk from $i$ to $j$ is directed if 
all edges are directed toward $j$. For example, in Figure \ref{fig:graph0} (A),
the walk $1 \leftarrow 4 \unEdge 4  \rightarrow 3 \rightarrow 5 \rightarrow 4$ 
is a walk from $1$ to $4$, and $3 \rightarrow 5 \rightarrow 5 \rightarrow 4 \rightarrow 1$ 
is a directed walk from $3$ to $1$.

We will by convention assume that a mixed graph includes all directed
self-loops unless otherwise stated. It will, however, also be convenient to have
a version of the graph with all directed self-loops removed, which we call the
\emph{base graph}.

\begin{definition}
    The base graph $\mathcal{G}_0$ of a mixed graph $\mathcal{G}$ is 
    obtained from $\mathcal{G}$ by removing all directed self-loops.
\end{definition}

In Figure \ref{fig:graph0} (B) we have the base graph of the mixed graph (A). A
walk in the base graph is called a base walk. The walk $3 \rightarrow 5
\rightarrow 5 \rightarrow 4 \rightarrow 1$ in Figure \ref{fig:graph0} (A) is not
a base walk due to the self-loop $5 \rightarrow 5$, while $1 \leftarrow 4
\unEdge 4  \rightarrow 3 \rightarrow 5 \rightarrow 4$ is a base walk.

\subsection{Linear additive noise models} It is instructive to compare and 
contrast our results for the Lyapunov equation with the well known results  
for linear additive noise models. In the linear additive noise model, the 
random variable $X \in \mathbb{R}^d$ is given as a solution to the equation
\begin{equation} \label{eq:LAN}
    X = B X + \varepsilon,
\end{equation}
where $B$ is a $d \times d$ matrix with $B_{ii} = 0$ and $\varepsilon$ has mean zero 
and covariance matrix $\Omega$. Consequently, the covariance matrix, 
$\Sigma$, of $X$ solves the equation 
\begin{equation} \label{eq:LANcov}
    (I - B) \Sigma (I - B)^T = \Omega.
\end{equation}
When $I - B$ is invertible, the solution is unique and given by $\Sigma =
(I - B)^{-1} \Omega (I - B)^{-T}$. As the Lyapunov equation, the
equation \eqref{eq:LANcov} is a linear matrix equation, and its solution yields a
parametrization of the covariance matrix in terms of the entries of $B$
and $\Omega$. Zero-constraints on the entries of $B$ and $\Omega$ can,
moreover, be represented by a mixed graph according to Definition
\ref{dfn:graph} just as for the Lyapunov equation. The covariance models 
entailed by the two equations are, however, quite different, though they 
also share some structural similarities.

\section{Treks and trek rules} \label{sec:treks}
The trek rule for the solution of
\eqref{eq:LANcov} is a well known power series representation of the entries
$\Sigma_{ij}$ in the covariance matrix in terms of the entries of $B$ and
$\Omega$. We derive in this section a corresponding trek rule for the solution
of the Lyapunov equation \eqref{eq:Lyapunov}. To this end, we need to define 
treks in a mixed graph.

\begin{definition}
    A trek $\tau$ in a mixed graph $\mathcal{G}$ is a walk of the form
    \begin{equation} \label{eq:trek}
        \tau: \ \underbrace{i \leftarrow \cdots \leftarrow i_1}_{n(\tau)} 
        \leftarrow i_0 \unEdge j_0 \rightarrow \underbrace{j_1
          \rightarrow \cdots \rightarrow j}_{m(\tau)}.
    \end{equation}
    Here $n(\tau)$ and $m(\tau)$ denote the number of nodes to the left
    of $i_0$ and to the right of $j_0$, respectively. Let $l(\tau) = n(\tau) + m(\tau)$.
\end{definition}

The trek given by \eqref{eq:trek} is said to be a trek from $i$ to $j$. 
The pair $(i_0, j_0)$ is the \emph{top} of the trek, 
and if $j_0 = i_0$ we refer to $i_0$ as the top. 
The nodes $i_0, \ldots, i$ and $j_0, \ldots, j$ are the \emph{left-hand}
and \emph{right-hand} sides of the trek, respectively. The top nodes are 
connected by a blunt edge, possibly a blunt self-loop, while all other edges are directed. 
The left-hand side of the trek forms a directed walk from $i_0$ to $i$ with 
$n(\tau) + 1$ nodes, and the right-hand side forms a directed walk from $j_0$ to $j$
with $m(\tau) + 1$ nodes. It is possible for a trek to have $n(\tau) = m(\tau) = 0$,
in which case the trek is just $i_0 \unEdge j_0$, which is possibly a 
blunt self-loop $i_0 \unEdge i_0$.

\begin{definition} Let $(M, C)$ be compatible with a mixed graph $\mathcal{G}$
    and let $\tau$ be a trek in $\mathcal{G}$ with top $(i_0, j_0)$. 
    The weight of the trek $\tau$ is the product of the edge weights
    along the trek, that is, 
    $$\omega(M, C, \tau) = c_{i_0, j_0} 
    \prod_{k \rightarrow l \in \tau} m_{lk}.$$
\end{definition}

The walk $1 \leftarrow 4 \unEdge 4  \rightarrow 3 \rightarrow 5 \rightarrow 4$ 
in Figure \ref{fig:graph0} (A) is an example of a trek $\tau$ from $1$ to $4$ with top $4$ 
and with weight $\omega(M, C, \tau) =  0.2 \cdot 1 \cdot 0.5 \cdot 1 \cdot 1 = 0.1$.

With the definitions above, suppose that $(B, \Omega)$ are compatible with a
mixed graph $\mathcal{G}$ and that $B$ has spectral radius less than $1$, then
the trek rule for the linear additive noise model, with $\Sigma$ the solution of
\eqref{eq:LANcov}, is 
\[
    \Sigma_{ij} = \sum_{\tau \in \mathcal{T}(i,j)} \omega(B, \Omega, \tau),
\]
where $ \mathcal{T}(i,j)$ denotes the set of all treks from $i$ to $j$ in
$\mathcal{G}$. See, e.g., Theorem 4.2 in \cite{drton2018} or Proposition 14.2.13 in
\cite{sullivant2018algebraic}. The proof is a simple application of the Neumann
series expansion of $(I - B)^{-1}$. 

For the Lyapunov equation, the
following variant of a trek rule was obtained as Proposition 2.2 in \cite{varando2020}. 
Suppose that $(M, C)$ is compatible with the mixed graph $\mathcal{G}$, and $M$ is
stable and $C$ is positive semidefinite. For any trek $\tau$ in $\mathcal{G}$
we introduce the monomial 
\[
    \kappa(s, \tau) = \frac{s^{l(\tau) + 1}}{((l(\tau) + 1)n(\tau)!m(\tau)!)}
\]
in the auxiliary variable $s \in \mathbb{R}$, and the solution of the Lyapunov
    equation \eqref{eq:Lyapunov} is then given by
    \begin{equation} \label{eq:trekrep}
    \Sigma_{ij} = \lim_{s \rightarrow \infty} 
    \sum_{\tau \in \mathcal{T}(i,j)} \kappa(s, \tau) \omega(M, C, \tau).
    \end{equation}

The series appearing in \eqref{eq:trekrep} is an infinite sum 
over all treks from $i$ to $j$ of the trek weights multiplied by
the factors $\kappa(s, \tau)$ (not depending on $M$ and $C$). Besides 
these factors, the series is similar to the classical trek rule. However, 
to obtain $\Sigma_{ij}$ we have to take the limit 
$s \rightarrow \infty$ of the resulting sum. This makes the representation
\eqref{eq:trekrep} somewhat clumsy and difficult to use and interpret.

The new trek rule that we give below essentially 
interchanges the summation and limit operation by a translation of 
$M$ by the identity matrix $I$ (and by possibly rescaling $M$). 
Throughout we will let
\begin{equation} \label{eq:Lambda}
    \Lambda = M + I.
\end{equation}
Since the Lyapunov equation is invariant to rescaling, 
we can w.l.o.g. assume that $\Lambda$ has spectral radius strictly smaller than 
$1$ when $M$ is stable. If not, we can always rescale $M$ and $C$ to ensure this
without changing $\Sigma$.

\begin{proposition}  \label{prop:trek}
    Let $(M, C)$ be compatible with a mixed graph $\mathcal{G}$, and let $M$ 
    be stable and $C$ be positive semidefinite. When $\Lambda = M + I$ 
    has spectral radius strictly smaller than $1$, the solution of the Lyapunov
    equation \eqref{eq:Lyapunov} is given by
    \begin{equation} \label{eq:trekrep2}
        \Sigma_{ij} = \sum_{\tau \in \mathcal{T}(i,j)} 
        2^{-l(\tau) - 1}
        {l(\tau) \choose n(\tau)}
        \omega(\Lambda, C, \tau),
    \end{equation}
     where $\mathcal{T}(i,j)$ denotes the set of all treks from $i$ to
    $j$ in $\mathcal{G}$. The convergence of the series \eqref{eq:trekrep2} is absolute.  
\end{proposition}

\begin{proof} Using the series expansion of the exponential 
    function we have
    \begin{align}
        \Sigma & = 
        \int_0^{\infty} e^{tM} C e^{tM^T} \mathrm{d} t \nonumber \\
        & = \int_0^{\infty} e^{-2t} e^{t\Lambda} C e^{t\Lambda^T} \mathrm{d} t \nonumber \\
        & = \int_0^{\infty} 
        \sum_{n=0}^{\infty} \sum_{m=0}^{\infty}                           
        \frac{e^{-2t} t^n t^m}{n!m!} \Lambda^n
                                  C (\Lambda^m)^T \mathrm{d} t. \label{eq:intsum}
    \end{align}
    With $\| \cdot\|$ denoting any matrix norm, the assumption that 
    the spectral radius of $\Lambda$ is strictly smaller than $1$, combined with 
    the spectral radius formula, implies the existence of constants $K > 0$ and $r \in [0, 1)$ such that
    \[
        \| \Lambda^n \| \leq K  r^n.
    \]
    This bound implies absolute convergence of \eqref{eq:intsum}, which justifies interchanging the summation and integration. Thus, 
    \begin{align*}
       \Sigma & = 
       \sum_{n=0}^{\infty} \sum_{m=0}^{\infty}  \Lambda^n
       C (\Lambda^m)^T \frac{1}{n!m!} \int_0^{\infty} t^{n + m} e^{-2t}\mathrm{d} t \\
        & =  
        \sum_{n=0}^{\infty} \sum_{m=0}^{\infty}  \Lambda^n
        C (\Lambda^m)^T \frac{2^{- n - m - 1}\Gamma(n + m + 1)}{n!m!}
        \\
        & =  
        \sum_{n=0}^{\infty} \sum_{m=0}^{\infty}  \Lambda^n
        C (\Lambda^m)^T  2^{- n - m - 1} {n + m \choose n}.
        \end{align*}

Since $\mathcal{G}$, by convention, includes all directed self-loops,
$(\Lambda, C)$ is also compatible with $\mathcal{G}$. Now note that 
$$(\Lambda^n C (\Lambda^m)^T)_{ij}$$
is precisely the sum over all trek weights, 
$\omega(\Lambda, C, \tau)$,
for treks $\tau$ from $i$ to $j$ with $n + 1$ nodes on the left-hand side   and 
$m + 1$ nodes on the right-hand side. By the arguments above, the series representation 
of $\Sigma$ is absolutely convergent and the summation order does not matter. 
Therefore, the $ij$-th entry of $\Sigma$ can be written as the trek representation \eqref{eq:trekrep2}.
\end{proof}

The trek weights are monomials in the $M$ and $C$ entries,
and \eqref{eq:trekrep2} provides a (generally infinite) power series representation 
of the covariances in terms of these entries. Some factors in the monomials 
are the diagonal entries $\Lambda_{ii} = m_{ii} + 1$, while the remaining 
factors are off-diagonal entries of $\Lambda$ (that coincide 
with off-diagonal entries of $M$) and entries of $C$. It is possible, 
as we will show, to derive a trek rule where the contributions from the 
the diagonal entries are separated from the off-diagonal entries.

Recall that the mixed graph $\mathcal{G}$ is assumed to include all directed 
self-loops, while the base graph $\mathcal{G}_0$ is
obtained by removing the directed self-loops. A base trek is then a trek in
the base graph $\mathcal{G}_0$, and it is a trek without any self-loops. Any 
trek in $\mathcal{G}$ can be regarded as a base trek combined with 
a total of $\alpha_1, \ldots, \alpha_d \geq 0$ self-loops on the 
left-hand side and $\beta_1, \ldots, \beta_d \geq 0$ self-loops on the 
right-hand side. Here $\alpha_i$ and $\beta_i$ denote the total 
number of self-loops for node $i$ present on the trek on the left- and 
right-hand side, respectively. 

To elaborate on the definition of the self-loop counts $\alpha_i$ and $\beta_i$,
consider the example trek 
\[
1 \leftarrow 1 \leftarrow 2 \leftarrow 2 \leftarrow 2 \leftarrow 1 \unEdge 1 \rightarrow 1 \rightarrow 2 \rightarrow 1  \rightarrow 1
\] 
with nodes $\{1, 2\}$. The corresponding base trek is
$1 \leftarrow 2 \leftarrow 1 \unEdge 1 \rightarrow 2  \rightarrow 1$, and the trek includes 
one self-loop at node $1$ and two self-loops at node $2$ on the left-hand side and two self-loops at node $1$ on the right-hand side, whence $\alpha_1 = 1$, $\alpha_2 = 2$, $\beta_1 = 2$, and $\beta_2 = 0$.

The base trek and self-loop counts do not uniquely determine a trek, but 
there are some constraints. If node $i$ is not present on the left-hand 
side of the base trek, then $\alpha_i = 0$. If node $i$ is present 
once on the left-hand side of the base trek, then $\alpha_i$ is exactly the 
number of self-loops at that position. In the example above, node $2$ is 
present once on the left-hand side of the base trek, and $\alpha_2 = 2$ 
gives the number of self-loops of node $2$ at that position.
If node $i$ is present more than once on the 
left-hand side of the base trek, then the total number of self-loops 
$\alpha_i$ can be partitioned in several ways among the different positions 
of node $i$, and similiarly for $\beta_i$ on the right-hand side.
This situation can only occur when the graph contains directed cycles beyond self-loops.
In the example above, node $1$ is 
present twice on the right-hand side of the base trek and $\beta_1 = 2$
can be partitioned in three different ways among the two occurrences of node $1$.

For a base trek $\tau \in \mathcal{G}_0$ and multiindices 
$\alpha, \beta \in \mathbb{N}_0^d$ we define $\rho(\tau, \alpha, \beta)$ 
as the number of ways the $\alpha$ and $\beta$ self-loops can be partitioned 
among the nodes of the base trek. If the base trek $\tau$ does not contain repeated nodes in 
its left- or right-hand sides, $\rho(\tau, \alpha, \beta) \in \{0, 1\}$, but 
when $\tau$ contains repeated nodes, it is possible that 
$\rho(\tau, \alpha, \beta) > 1$. Define also 
$$\alpha_{\bullet} = \sum_{i=1}^d \alpha_i$$
and similarly for $\beta$.

\begin{definition} \label{dfn:D} For $\lambda_1, \ldots, \lambda_d \in (-1, 1)$ and a 
    base trek $\tau$ define
    \begin{equation} \label{eq:genD}
    D(\lambda_1, \ldots, \lambda_d, \tau) = \sum_{\alpha, \beta \in \mathbb{N}_0^d}
    \rho(\tau, \alpha, \beta)
    {l(\tau) + \alpha_{\bullet} + \beta_{\bullet} \choose n(\tau) + \alpha_{\bullet}} 
    \prod_{i=1}^d \left(\frac{\lambda_i}{2}\right)^{\alpha_i + \beta_i}.
    \end{equation}
\end{definition}

It may not be obvious that the series in \eqref{eq:genD} converges. Its 
absolute convergence does, however, follow from Theorem \ref{thm:altrep} below
by taking $C = I$ and letting $m_{ii} \in [-1, 0)$ and $m_{ij} = \varepsilon$ 
for $i \neq j$ for a small $\varepsilon > 0$. Then $M + I$ has spectral radius 
strictly smaller than $1$ by Perron-Frobenius theory, and 
$D(m_{11} + 1, \ldots, m_{dd} + 1, \tau) < \infty$ for all base treks $\tau$
in the complete graph since \eqref{eq:trekrep4} below is finite and 
all terms in that sum are positive.

\begin{theorem} \label{thm:altrep} 
    Let $(M, C)$ be compatible with a mixed graph $\mathcal{G}$, and let $M$ 
    be stable and $C$ be positive semidefinite. If $\Lambda = M + I$ has spectral radius strictly smaller
    than $1$ and \mbox{$\Lambda_{ii} = m_{ii} + 1\in (-1, 1)$}, then
    \begin{align} \label{eq:trekrep3}
        \Sigma_{ij} & = \sum_{\tau \in \mathcal{T}_0(i,j)} 
        2^{-l(\tau) - 1} D(\Lambda_{11}, \ldots, \Lambda_{dd}, \tau)
        \omega(\Lambda, C, \tau) \\
        \label{eq:trekrep4}
        & = \sum_{\tau \in \mathcal{T}_0(i,j)} 
        2^{-l(\tau) - 1} D(m_{11} + 1, \ldots, m_{dd} + 1, \tau)
        \omega(M, C, \tau),
    \end{align}
    where $\mathcal{T}_0(i,j)$ denotes the set of all treks from $i$ to
    $j$ in the base graph $\mathcal{G}_0$ of $\mathcal{G}$. In the special case 
    where $m_{ii} = -1$ for all $i$, then 
    \begin{equation} \label{eq:trekrep5}
        \Sigma_{ij} = \sum_{\tau \in \mathcal{T}_0(i,j)} 
        2^{-l(\tau) - 1} {l(\tau) \choose n(\tau)}
        \omega(M, C, \tau).
    \end{equation}
\end{theorem}

\begin{proof} Since the convergence of the series \eqref{eq:trekrep2} is absolute, we can reorder  the terms.    
    The formula \eqref{eq:trekrep3} follows by splitting the sum \eqref{eq:trekrep2} into 
    an outer sum over the base treks and an inner sum over self-loops. 
    Specifically, 
    \begin{align*}
        \Sigma_{ij} & = \sum_{\tau \in \mathcal{T}(i,j)} 
        2^{-l(\tau) - 1} {l(\tau) \choose n(\tau)}
        \omega(\Lambda, C, \tau) \\
        & = \sum_{\tau \in \mathcal{T}_0(i,j)} 
        \sum_{\alpha, \beta \in \mathbb{N}_0^d}
        \rho(\tau, \alpha, \beta) 
        2^{-l(\tau) - \alpha_{\bullet} - \beta_{\bullet}- 1} 
        {l(\tau) + \alpha_{\bullet} + \beta_{\bullet} \choose n(\tau) + \alpha_{\bullet}}
        \prod_{i=1}^d \Lambda_{ii}^{\alpha_i + \beta_i}  \omega(\Lambda, C, \tau) \\
        & = \sum_{\tau \in \mathcal{T}_0(i,j)} 
        2^{-l(\tau) - 1} \left(\sum_{\alpha, \beta \in \mathbb{N}_0^d}
        \rho(\tau, \alpha, \beta)
        {l(\tau)  + \alpha_{\bullet} + \beta_{\bullet}  \choose n(\tau)  + \alpha_{\bullet}  } 
        \prod_{i=1}^d \left(\frac{\Lambda_{ii}}{2}\right)^{\alpha_i + \beta_i}\right)  
        \omega(\Lambda, C, \tau) \\
        & = \sum_{\tau \in \mathcal{T}_0(i,j)} 
        2^{-l(\tau) - 1} D(\Lambda_{11}, \ldots, \Lambda_{dd}, \tau)
        \omega(\Lambda, C, \tau).
    \end{align*}
    Moreover, \eqref{eq:trekrep4} follows by the definition of
    $\Lambda_{ii} = m_{ii} + 1$ and by noting that  $\omega(\Lambda, C, \tau)
    =   \omega(M, C, \tau)$ for any base trek $\tau$.
    Finally, if $m_{ii} = -1$ then $\Lambda_{ii} = m_{ii} + 1 = 0$ for 
    $i = 1, \ldots, d$, and since 
    $$D(0, \ldots, 0, \tau) = {l(\tau) \choose n(\tau)}$$ 
    for all base treks $\tau$, \eqref{eq:trekrep5} follows from \eqref{eq:trekrep4}.
\end{proof}

Note that for a base trek, the monomial 
$\omega(M, C, \tau)$ does not depend on the 
diagonal elements of $M$, and \eqref{eq:trekrep4} 
provides a certain disentanglement of how the total 
covariance depends on the self-loop entries 
and the other edge entries.


\begin{ex} \label{ex:graph0_cont} To illustrate the trek rule we revisit
Example \ref{ex:graph0}, specifically the entry $\Sigma_{13} = 0.123$.
Note that $m_{ii} = -1$ for all $i$. 
Table \ref{tab:treks} lists all 12 base treks from $1$ to $3$ in the base graph (B) in Figure 
\ref{fig:graph0} with $l(\tau) \leq 5$ together with their corresponding 
trek weights and term values contributing to $\Sigma_{13}$ in the sum \eqref{eq:trekrep5}. 

\setlength{\tabcolsep}{4.5pt} 
\renewcommand{\arraystretch}{1.45} 

\begin{table}[h]
    \centering
    {\footnotesize 
    \begin{tabular}{lrlcclr}
         Trek & $\omega$ & $\omega$ factorization  &  $l$ &  $n$ & 
         $2^{-l - 1} {l \choose n} \cdot \omega$ & Term value \\
        \hline
        $1 \leftarrow 4 \unEdge 4 \rightarrow 3$ & $0.1$ & $0.2 \cdot 0.5$ & 
        $2$ & $1$ & $2^{-3} \cdot {2 \choose 1} \cdot 0.1$ & $0.02500000$  \\
        
        $1 \leftarrow 2 \leftarrow 3 \unEdge 3$ & $0.1$ & $0.5 \cdot 0.2$ & 
        $2$ & $2$ & $2^{-3} \cdot {2 \choose 2} \cdot 0.1$ & $0.01250000$  \\
        
        \hline
        
        $1 \leftarrow 4 \leftarrow 5 \leftarrow 3 \unEdge 3$ & 
        $0.2$ & $0.2 \cdot 1 \cdot 1$ & 
        $3$ & $3$ & $2^{-4} \cdot {3 \choose 3} \cdot 0.2$ & $0.01250000$  \\
        
        \hline
        
        $1 \leftarrow 2 \leftarrow 1 \leftarrow 2 \leftarrow 3 \unEdge 3$ & 
        $-0.05$ & $0.5 \cdot (- 1) \cdot 0.5 \cdot 0.2$ & 
        $4$ & $4$ & $2^{-5} \cdot {4 \choose 4} \cdot (-0.05)$ & $-0.00156250$  \\
        
        $1 \leftarrow 2 \leftarrow 1 \leftarrow 4 \unEdge 4 \rightarrow 3$ & 
        $-0.05$ & $0.5 \cdot (-1) \cdot 0.2 \cdot 0.5$ & 
        $4$ & $3$ & $2^{-5} \cdot {4 \choose 3} \cdot (-0.05)$ & $-0.00625000$  \\

        $1 \leftarrow 2 \leftarrow 3 \leftarrow 4 \unEdge 4 \rightarrow 3$ & 
        $0.025$ & $0.5 \cdot 0.2 \cdot 0.5 \cdot 0.5$ & 
        $4$ & $3$ & $2^{-5} \cdot {4 \choose 3} \cdot 0.025$ & $0.00312500$  \\
        
        $1 \leftarrow 4 \leftarrow 5 \unEdge 5 \rightarrow 4 \rightarrow 3$ & 
        $0.1$ & $0.2 \cdot 1 \cdot 1 \cdot 0.5$ & 
        $4$ & $2$ & $2^{-5} \cdot {4 \choose 2} \cdot 0.1$ & $0.01875000$  \\

        \hline
        
        $1 \leftarrow 2 \leftarrow 3 \leftarrow 4 \leftarrow 5 \leftarrow 
        3 \unEdge 3$ & 
        $0.05$ & $0.5 \cdot 0.2 \cdot 1 \cdot 1 \cdot 0.5$ & 
        $5$ & $5$ & $2^{-6} \cdot {5 \choose 5} \cdot 0.05$ & $0.00078125$  \\

        $1 \leftarrow 2 \leftarrow 1 \leftarrow 4 \leftarrow 5 \leftarrow 
        3 \unEdge 3$ & 
        $-0.1$ & $0.5 \cdot (-1) \cdot 0.2 \cdot 1 \cdot 1$ & 
        $5$ & $5$ & $2^{-6} \cdot {5 \choose 5} \cdot (-0.1)$ & $-0.00156250$  \\
        
        $1 \leftarrow 4 \leftarrow 5 \leftarrow 3 \leftarrow 4 \unEdge 4 \rightarrow 3$ & 
        $0.05$ & $0.2 \cdot 1 \cdot 1 \cdot 0.5 \cdot 0.5$ & 
        $5$ & $4$ & $2^{-6} \cdot {5 \choose 4} \cdot 0.05$ & $0.00390625$  \\

        $1 \leftarrow 2 \leftarrow 3 \unEdge 3 \rightarrow 5 \rightarrow 4 \rightarrow 3$ & 
        $0.05$ & $0.5 \cdot 0.2 \cdot 1 \cdot 1 \cdot 0.5$ & 
        $5$ & $2$ & $2^{-6} \cdot {5 \choose 2} \cdot 0.05$ & $0.00781250$  \\
        
        $1 \leftarrow 4 \unEdge 4 \rightarrow 3 \rightarrow 5 \rightarrow 4 
        \rightarrow 3$ & 
        $0.05$ & $0.2 \cdot 0.5 \cdot 1 \cdot 1 \cdot 0.5$ & 
        $5$ & $1$ & $2^{-6} \cdot {5 \choose 1} \cdot 0.05$ & $0.00390625$  \\
        \hline \hline 
        Total & & & & & & $0.07890625$ \\
    \end{tabular}
    }
    \caption{Base treks $\tau$ with $l = l(\tau) \leq 5$ in Figure \ref{fig:graph0} (B)
    and their corresponding weights $\omega = \omega(M, C, \tau)$ and term values
    $2^{-l - 1} {l \choose n} \cdot \omega$ in the sum \eqref{eq:trekrep5}.
    \label{tab:treks}}
\end{table}

\setlength{\tabcolsep}{6pt} 
\renewcommand{\arraystretch}{1} 

We see from Table \ref{tab:treks} that 
\[
    \sum_{\tau \in \mathcal{T}_0(1,3): l(\tau) \leq 5} 
    2^{-l(\tau) - 1} {l(\tau) \choose n(\tau)} \omega(M, C, \tau) 
    = 0.07890625,
\]
which is still a bit from the limit value $\Sigma_{13} = 0.123$. Proceeding up
to $l(\tau) = 10$ gives $74$ base treks, and the sum of the corresponding terms
is $0.10992737$, while the sum of the $515$ terms with $l(\tau) \leq 20$ is
$0.12127330$, which is accurate up to the second decimal. The purpose of this
example \emph{is not} to claim that the trek rule is useful for computing the
solution of the Lyapunov equation in the general case. On the contrary, there is
quite a lot of bookkeeping involved in computing the treks and the corresponding
terms in the sum, and you may need a fairly large number of terms to get an
accurate finite sum approximation. The purpose of the example is rather to
illustrate how the trek rule breaks down the total covariance into contributions
from the individual treks, with each term being a monomial in the edge
entries.
\end{ex}

\section{Acyclic models} \label{sec:acyclic}

If the directed part of the base graph $\mathcal{G}_0$ is acyclic, 
and thus a DAG, we say that the model given by $M$ and $C$ is acyclic. Choosing any 
topological order of the nodes will make the $M$-matrix lower triangular.
We assume throughout this section that the model is acyclic and that 
the nodes $1, \ldots, d$ are ordered in a topological order such that
$M$ is lower triangular, that is,
\[
    M = \left( \begin{array}{cccccc} 
        m_{11} & . & . & \ldots & . \\
        m_{21} & m_{22} & . & \ldots & . \\
        m_{31} & m_{32} & m_{33} & \ldots  & . \\
        \vdots & \vdots & \vdots & \ddots  & \vdots \\
        m_{d1} & m_{d2} & m_{d3} & \ldots &  m_{dd}        
    \end{array}\right). 
\]

The diagonal entries of $M$ are then the eigenvalues, and $M$ is stable 
if and only $m_{ii} < 0$ for all $i$. By rescaling, we can assume that
$m_{ii} \in [-1, 0)$ for all $i$ so that $\Lambda_{ii} = m_{ii} + 1 \in [0, 1)$,
in which case $\Lambda$ also has spectral radius strictly less than $1$.

\subsection{Simplifying the trek rule for acyclic models}
For an acyclic model any base trek has no nodes repeated on either side 
(though the 
same node can be present once on the left- and once on the right-hand side).
This means that $\rho(\tau, \alpha, \beta) \in \{0, 1\}$, and if we 
define 
$$\mathcal{N}(\tau) = \{ (\alpha, \beta) \in \mathbb{N}_0^d \times \mathbb{N}_0^d 
\mid \rho(\tau, \alpha, \beta) = 1\}$$
then 
\begin{equation} \label{eq:triD}
    D(m_{11} + 1, \ldots, m_{dd} + 1, \tau) = \sum_{(\alpha, \beta) \in \mathcal{N}(\tau)}
    {l(\tau) + \alpha_{\bullet} + \beta_{\bullet} \choose n(\tau) + \alpha_{\bullet}} 
    \prod_{i=1}^d \left(\frac{m_{ii} + 1}{2}\right)^{\alpha_i + \beta_i}.
\end{equation}
The expression \eqref{eq:triD} does not appear to simplify further in general.
However, if all the diagonal entries of $M$ are equal, in which case we can
assume them all equal to $-1$ by rescaling, the trek rule simplifies to
\eqref{eq:trekrep5}. This special case is effectively the same as Proposition 
 4.3 in \cite{boege2024}, which was obtained directly by an induction argument.

For any acyclic model, the directed part of $\mathcal{G}_0$ forms a DAG, and
there are no loops but self-loops in $\mathcal{G}$. The number of base treks
from $i$ to $j$ is thus finite, and \eqref{eq:trekrep4} gives a
\emph{finite} sum representation of $\Sigma_{ij}$. That is, the representation
\eqref{eq:trekrep4} is a polynomial in the off-diagonal entries of $M$ and the
entries of $C$.

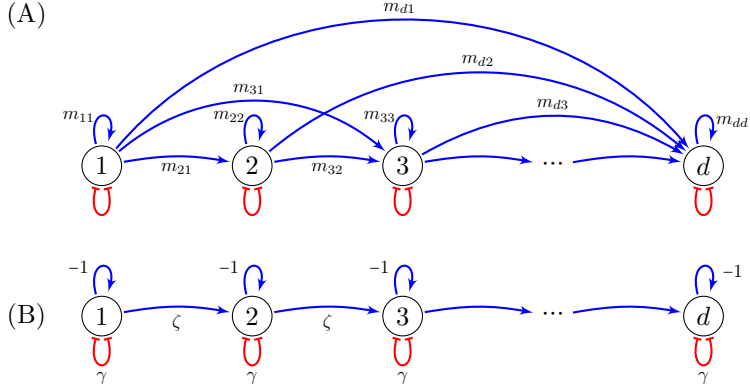
\begin{figure}
    \centering
    \begin{tikzpicture}
    \tikzset{vertex/.style = {shape=circle,draw,minimum size=1.5em, inner sep = 0pt, outer sep = 1pt}}
    \tikzset{vertexdot/.style = {shape=circle,fill=red,color=red,draw,minimum size=0.5em, inner sep = 0pt}}
    \tikzset{vertexhid/.style = {shape=rectangle,draw,minimum size=1.5em, inner sep = 0pt}}
    \tikzset{edge/.style = {->,> = latex', thick}}
    \tikzset{edgeun/.style = {-, thick}}
    \tikzset{edgebi/.style = {<->,> = latex', thick}}
    \tikzset{edgeblunt/.style = {{|[scale=0.5]}-{|[scale=0.5]}, thick}}

    \node at (-4, 2) {(A)};
    \node at (-4, -2) {(B)};
    
    \node[vertex] (A) at (-3, 0) {$1$};
    \node[vertex] (B) at (-1, 0) {$2$};
    \node[vertex] (C) at (1, 0) {$3$};
    \node (D) at (3, 0) {$\cdots$};
    \node[vertex] (E) at (5, 0) {$d$};
    
    \draw[edge] (A) edge[loop above, color=blue] node[left, color=black, scale=0.7] {$m_{11}$} (A);
    \draw[edge] (A) edge[bend left = 10, color=blue] node[below, color=black, scale=0.7] {$m_{21}$} (B);
    \draw[edge] (A) edge[bend left = 40, color=blue]  node[above, pos=0.5, color=black, scale=0.7] {$m_{31}$ \ } (C);
    \draw[edge] (A) edge[bend left = 50, color=blue]  node[above, pos=0.5, color=black, scale=0.7] {$m_{d1}$ \ } (E);
    \draw[edge] (B) edge[bend left  = 10, color=blue] node[below, color=black, scale=0.7] {$m_{32}$} (C);
    \draw[edge] (B) edge[bend left  = 40, color=blue] node[above, color=black, scale=0.7] {$m_{d2}$} (E);
    \draw[edge] (B) edge[loop above, color=blue] node[left, color=black, scale=0.7] {$m_{22}$} (B);
    \draw[edge] (C) edge[loop above, color=blue] node[left, color=black, scale=0.7] {$m_{33}$} (C);
    \draw[edge] (C) edge[bend left  = 10, color=blue]  (D);
    \draw[edge] (C) edge[bend left  = 30, color=blue] node[above, color=black, scale=0.7] {$m_{d3}$} (E);
    \draw[edge] (D) edge[bend left  = 10, color=blue]  (E);
    \draw[edge] (E) edge[loop above, color=blue] node[right, pos=0.7, color=black, scale=0.7] {$m_{dd}$} (E);
    
    \draw[edgeblunt, loop below, color=red] (A) to (A);
    \draw[edgeblunt, loop below, color=red] (B) to (B);
    \draw[edgeblunt, loop below, color=red] (C) to (C);
    \draw[edgeblunt, loop below, color=red] (E) to (E);

    \node[vertex] (A1) at (-3, -2) {$1$};
    \node[vertex] (B1) at (-1, -2) {$2$};
    \node[vertex] (C1) at (1, -2) {$3$};
    \node (D1) at (3, -2) {$\cdots$};
    \node[vertex] (E1) at (5, -2) {$d$};
    
    \draw[edge] (A1) edge[loop above, color=blue] node[left, color=black, scale=0.7] {$-1\ $} (A1);
    \draw[edge] (A1) edge[bend left = 10, color=blue] node[below, color=black, scale=0.7] {$\zeta$} (B1);
    \draw[edge] (B1) edge[bend left  = 10, color=blue] node[below, color=black, scale=0.7] {$\zeta$} (C1);
    \draw[edge] (B1) edge[loop above, color=blue] node[left, color=black, scale=0.7] {$-1\ $} (B1);
    \draw[edge] (C1) edge[loop above, color=blue] node[left, color=black, scale=0.7] {$-1\ $} (C1);
    \draw[edge] (C1) edge[bend left  = 10, color=blue]  (D1);
    \draw[edge] (D1) edge[bend left  = 10, color=blue]  (E1);
    \draw[edge] (E1) edge[loop above, color=blue] node[right, pos=0.7, color=black, scale=0.7] {$\ -1$} (E1);
    
    \draw[edgeblunt, loop below, color=red] (A1) to node[below, color=black, scale=0.7] {$\gamma$} (A1);
    \draw[edgeblunt, loop below, color=red] (B1) to node[below, color=black, scale=0.7] {$\gamma$} (B1);
    \draw[edgeblunt, loop below, color=red] (C1) to node[below, color=black, scale=0.7] {$\gamma$} (C1);
    \draw[edgeblunt, loop below, color=red] (E1) to node[below, color=black, scale=0.7] {$\gamma$} (E1);
    
    \end{tikzpicture} 
    \caption{The mixed graph (A) for a general acyclic model with the nodes in 
        a topological order, and the mixed graph (B) for the specific model in Example \ref{ex:path}. 
        \label{fig:graph1}}
\end{figure}

\begin{ex}[Path model] \label{ex:path}
    Suppose $m_{ii} = -1$, $m_{(i+1)i} = \zeta \in \mathbb{R}$ and $c_{ii} = \gamma \geq 0$. 
    All other entries are 0.
    That is
    \[
    M = \left( \begin{array}{ccccccc} 
        -1 & . & . & \ldots & . & . \\
        \zeta & -1 & . & \ldots & . & . \\
        . & \zeta & -1 & \ldots & . & . \\
        \vdots & \vdots & \vdots & \ddots & \vdots & \vdots \\
        . & . & . & \ldots & -1 & . \\        
        . & . & . & \ldots & \zeta & -1        
    \end{array}\right) \quad \text{and} \quad
    C = \left( \begin{array}{ccccccc} 
        \gamma & . & . & \ldots & . & . \\
        . & \gamma & . & \ldots & . & .\\
        . & . & \gamma & \ldots & . & . \\
        \vdots & \vdots & \vdots & \ddots  & \vdots & \vdots \\
        . & . & . & \ldots & \gamma & . \\
        . & . & . & \ldots & . & \gamma
    \end{array}\right).
    \]
    The mixed graph for this model is shown in Figure \ref{fig:graph1} (B).
    We have $\Lambda_{(i+1)i} = \zeta$ and all 
    other entries of $\Lambda$ are 0. The only base treks 
    from $i$ to $j$ have top $i_0 \leq \min\{i,j\}$ and are of the form
    $$\tau: i \leftarrow i - 1 \leftarrow \ldots i_0 + 1 \leftarrow i_0 
    \unEdge i_0 \rightarrow i_0 + 1 \rightarrow \ldots \rightarrow j - 1 \rightarrow j,$$
    for which $n(\tau) = i - i_0$, $m(\tau) = j - i_0$, $l(\tau) = i + j - 2i_0$.
    This shows that 
    \begin{align} 
    \Sigma_{ij} & = \sum_{i_0=1}^{\min\{i,j\}} 2^{2i_0 - i - j - 1}
     {i + j - 2i_0 \choose i - i_0} \zeta^{i + j - 2 i_0} \gamma  \nonumber \\
     & =  \frac{\gamma}{2}  \left(\frac{\zeta}{2}\right)^{i + j} 
     \sum_{i_0=1}^{\min\{i,j\}} \left(\frac{4}{\zeta^2}\right)^{i_0}
     {i + j - 2i_0 \choose i - i_0}. \label{eq:Sigmaij}
    \end{align}
    We see that 
    $\Sigma_{i1} = \Sigma_{1i} = \gamma / 2 (\zeta / 2)^{i-1}$, and by induction 
    the following recursion holds for $i,j > 1$ :
    \begin{equation} \label{eq:recursion}
        \Sigma_{ij} = \frac{\zeta}{2} \Sigma_{i-1,j} + 
        \frac{\zeta}{2}  \Sigma_{i,j-1} + \frac{\gamma}{2} \delta_{ij},
    \end{equation}
    where $\delta_{ij}$ is the Kronecker delta. 

    If we take $\zeta = \gamma = 2$, we get the simple expression 
    \begin{equation} \label{eq:Sigmaij2}
        \Sigma_{ij} = \sum_{i_0=1}^{\min\{i,j\}} 
        {i + j - 2i_0 \choose i - i_0},
    \end{equation}
    involving only a sum of binomial coefficients. Then $\Sigma_{i1} = \Sigma_{1i} = 1$,
    the recursion is 
    \[
        \Sigma_{ij} =  \Sigma_{i-1,j} + 
        \Sigma_{i,j-1} + \delta_{ij},
    \]
    and the corresponding covariance matrix becomes  
    \begin{equation}  \label{eq:pascal}
        \Sigma = \left( \begin{array}{cccccc} 
            1 & 1 & 1 & 1 & 1 & \ldots \\
            1 & 3 & 4 & 5 & 6 & \ldots \\
            1 & 4 & 9 & 14 & 20 & \ldots \\
            1 & 5 & 14 & 29 & 49 & \ldots \\
            1 & 6 & 20 & 49 & 99 & \ldots \\
            \vdots & \vdots & \vdots & \vdots & \vdots & \ddots
        \end{array}\right)
    \end{equation}
    with the anti-diagonal elements being similar to Pascal's triangle 
    except that an extra 1 is added to the diagonal. 
    This is \href{https://oeis.org/A013580}{OEIS A013580}. 
\end{ex}

\begin{ex}[Factor model] \label{ex:factor} Another simple 
    acyclic model is the factor model with $m_{11} = -1$, $m_{ii} \in [-1, 0)$ for 
    $i = 2, \ldots, d$, $C$ diagonal, and the only possible off-diagonal 
    non-zero entries of $M$ being $m_{i1}$ for $i = 2, \ldots, d$. That is, 
    {\small \[
    M = \left( \begin{array}{cccccc} 
        -1 & . & . & \ldots &  . \\
        m_{21} & m_{22} & . & \ldots & . \\
        m_{31} & . & m_{33} & \ldots & . \\
        \vdots & \vdots & \vdots & \ddots & \vdots \\
        m_{d1} & . & . & \ldots  & m_{dd}
    \end{array}\right) \quad \text{and} \quad
    C = \left( \begin{array}{cccccc} 
        c_{11} & . & . & \ldots & . \\
        . & c_{22} & . & \ldots & .\\
        . & . & c_{33} & \ldots  & . \\
        \vdots & \vdots & \vdots & \ddots & \vdots \\
        . & . & . & \ldots & c_{dd}
    \end{array}\right).
    \] \par}
    The mixed graph for this model is shown in Figure \ref{fig:graph2}.
    The only base treks are of the form $i \unEdge i$ and 
    $ i \leftarrow 1 \unEdge 1 \rightarrow j$ for $i, j > 1$. 
    Using \eqref{eq:trekrep4} we get for $i, j > 1$ that 
    \begin{equation} \label{eq:factor_off}
        \Sigma_{ij} = \left\{ 
            \begin{array}{ll} 
                d^0_i c_{ii} +  d_{ii} m_{i1}^2 c_{11} & \quad \text{if } i = j \\
                d_{ij} m_{i1}m_{j1} c_{11} & \quad \text{if } i \neq j.
            \end{array}
        \right.
    \end{equation}
    where 
    \begin{align*}
    d^0_{i} & = 2^{-1} D(0, m_{22} + 1, \ldots, m_{dd} + 1, i \unEdge i) \\
    d_{ij} & = 2^{-3} D(0, m_{22} + 1, \ldots, m_{dd} + 1, i \leftarrow 1 \unEdge 1 \rightarrow j).
    \end{align*}
A somewhat tedious computation, similar to the one given in the proof of Proposition \ref{prop:lowerbound},
shows that $d^0_{i} = - 1/(2m_{ii})$ and
    \begin{equation} \label{eq:dij}
        d_{ij} = \frac{1}{2} \frac{m_{ii} + m_{jj} - 2}{(1 - m_{ii})(1 - m_{jj})(m_{ii} + m_{jj})} > 0.    
    \end{equation}

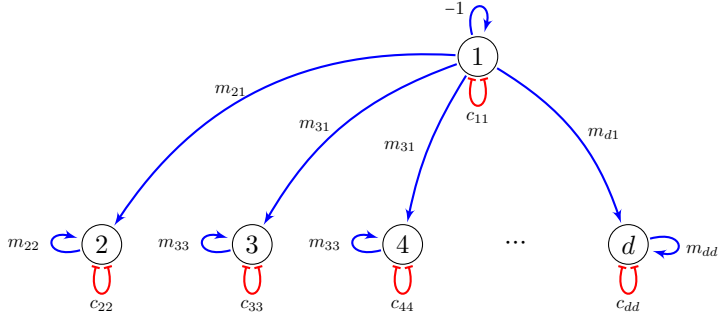
\begin{figure}
    \centering
    \begin{tikzpicture}
    \tikzset{vertex/.style = {shape=circle,draw,minimum size=1.5em, inner sep = 0pt, outer sep = 1pt}}
    \tikzset{vertexdot/.style = {shape=circle,fill=red,color=red,draw,minimum size=0.5em, inner sep = 0pt}}
    \tikzset{vertexhid/.style = {shape=rectangle,draw,minimum size=1.5em, inner sep = 0pt}}
    \tikzset{edge/.style = {->,> = latex', thick}}
    \tikzset{edgeun/.style = {-, thick}}
    \tikzset{edgebi/.style = {<->,> = latex', thick}}
    \tikzset{edgeblunt/.style = {{|[scale=0.5]}-{|[scale=0.5]}, thick}}
    
    \node[vertex] (A) at (4, 2.5) {$1$};
    \node[vertex] (B) at (-1, 0) {$2$};
    \node[vertex] (C) at (1, 0) {$3$};
    \node[vertex] (D) at (3, 0) {$4$};
    \node at (4.5, 0) {$\cdots$};
    \node[vertex] (E) at (6, 0) {$d$};
    
    \draw[edge] (A) edge[loop above, color=blue] node[left, color=black, scale=0.7] {$-1\ $} (A);
    \draw[edge] (B) edge[loop left, color=blue] node[left, color=black, scale=0.7] {$m_{22}\ $} (B);
    \draw[edge] (C) edge[loop left, color=blue] node[left, color=black, scale=0.7] {$m_{33}\ $} (C);
    \draw[edge] (D) edge[loop left, color=blue] node[left, color=black, scale=0.7] {$m_{33}\ $} (C);
    \draw[edge] (E) edge[loop right, color=blue] node[right, pos=0.7, color=black, scale=0.7] {$\  m_{dd}$} (E);

    \draw[edge] (A) edge[bend right = 30, color=blue] node[left, pos = 0.5, color=black, scale=0.7] {$m_{21}\ $} (B);
    \draw[edge] (A) edge[bend right = 20, color=blue]  node[left, pos=0.5, color=black, scale=0.7] {$m_{31}\ $} (C);
    \draw[edge] (A) edge[bend right = 10, color=blue]  node[left, pos=0.5, color=black, scale=0.7] {$m_{31}\ $ } (D);
    \draw[edge] (A) edge[bend left = 20, color=blue]  node[right, pos=0.5, color=black, scale=0.7] {$\ m_{d1}$} (E);
   
    \draw[edgeblunt, loop below, color=red] (A) to node[below, color=black, scale=0.7] {$c_{11}$} (A);
    \draw[edgeblunt, loop below, color=red] (B) to node[below, color=black, scale=0.7] {$c_{22}$} (B);
    \draw[edgeblunt, loop below, color=red] (C) to node[below, color=black, scale=0.7] {$c_{33}$} (C);
    \draw[edgeblunt, loop below, color=red] (D) to node[below, color=black, scale=0.7] {$c_{44}$} (D);
    \draw[edgeblunt, loop below, color=red] (E) to node[below, color=black, scale=0.7] {$c_{dd}$} (E);

    \end{tikzpicture} 
    \caption{The mixed graph for the factor model in Example \ref{ex:factor}. 
        \label{fig:graph2}}
\end{figure}

Another way to write the off-diagonal entries is
\begin{align} 
    \label{eq:factor_off2}
    \Sigma_{ij} & = \left(\frac{1}{2} - \frac{1}{m_{ii} + m_{jj}}\right) 
    \frac{m_{i1}}{1 - m_{ii}} \frac{m_{j1}}{1 - m_{jj}} c_{11} \qquad i, i > 1, i \neq j,
\end{align}
which reveals that the model does not generally have the same low-rank structure as 
the classical one-factor model based on linear additive noise. Indeed, the 
tetrad $\Sigma_{ij} \Sigma_{kl} - \Sigma_{il}\Sigma_{kj}$ is non-zero 
unless $(m_{ii} + m_{jj})(m_{kk} + m_{ll}) = (m_{ii} + m_{ll})(m_{jj} + m_{kk})$,
see \cite{Drton:2007aa,drton2024} for further details on algebraic invariants 
for factor models. It is an open problem to characterize invariants for factor 
models based on the Lyapunov equation.
\end{ex}

\subsection{A lower bound on the marginal variance} 

Consider an \emph{acyclic} linear additive noise model with the nodes in a topological order
such that 
\[
 B = \left( \begin{array}{cccccc} 
    0 & . & . & \ldots & . \\
    \beta_{21} & 0 & . & \ldots & . \\
    \beta_{31} & \beta_{32} & 0 & \ldots & . \\
    \vdots & \vdots & \vdots & \ddots & \vdots \\
    \beta_{d1} & \beta_{d2} & \beta_{d3} & \ldots & 0   
\end{array}\right) = 
\left( \begin{array}{cc} 
    B_{11} & 0 \\
    B_{d1} & 0 
\end{array}\right),
\]
where $B_{11}$ is the $(d-1) \times (d-1)$ upper left block of $B$ and $B_{d1} = (\beta_{d1}, \ldots, \beta_{d(d-1)})$ is the $d-1$ first entries of the 
last row of $B$. It then 
follows directly from \eqref{eq:LAN} that for this linear additive noise model,
\[
    X_d = \sum_{i=1}^{d-1} B_{di} X_i + \varepsilon_d = B_{d1} (X_1, \ldots, X_{d-1})^T + \varepsilon_d.
\]  
Hence, since $\varepsilon_d$ is independent of $(X_1, \ldots, X_{d-1})^T$ and 
$\Omega_{dd} = \text{Var}(\varepsilon_d)$, the marginal variance of $X_d$ is 
\begin{equation} \label{eq:LAN2}
    \Sigma_{dd} =  \Omega_{dd} + B_{d1} \bSigma_{11} (B_{d1})^T,
\end{equation}
where $\bSigma_{11}$ is the
covariance matrix of $(X_1, \ldots, X_{d-1})^T$. 

The marginal variance of the last node $d$ is by \eqref{eq:LAN2} decomposed 
into the
variance of the residual error $\varepsilon_d$ and the variance propagated
forward from all parents of $d$. As these parents have variances that again
decompose into residual error variances and the variances propagated from their
parents, the marginal variances $\Sigma_{ii}$ tend to increase along the
topological order of the nodes. \cite{reisach2021beware}
introduced \emph{varsortability} as a measure that captures this phenomenon, and they
demonstrated that commonly used simulation models typically have marginal
variances that increase along the topological order. In this case, the
topological order can be identified from the order of the marginal variances.
See also \cite{Park:2020} for related ideas. 

In this section we explore if a similar phenomenon holds for the solution of the
Lyapunov equation in the acyclic case, and we will, in particular, investigate
if the variance $\Sigma_{dd}$ has a decomposition similar to \eqref{eq:LAN2}.
Specifically, we derive a lower bound on $\Sigma_{dd}$ in terms of the other
(co)variances, which is similar to the identity \eqref{eq:LAN2}. To state the
result, we write the matrices $M$ and $\Sigma$ in block form as
\begin{align*}
    M = \begin{pmatrix}
        M_{11} & 0\\
        M_{d1} & m_{dd}
    \end{pmatrix} \quad \text{and} \quad
    \Sigma = \begin{pmatrix}
        \bSigma_{11} & \bSigma_{1d} \\
        \bSigma_{d1} & \Sigma_{dd}
    \end{pmatrix},
\end{align*}
where $M_{11}$ is a $(d-1) \times (d-1)$ matrix, 
 and $M_{d1}$ is a $(d-1)$-dimensional row vector, and similarly for
$\Sigma$.

\begin{proposition} \label{prop:lowerbound}
    Let $M$ be a lower triangular matrix with $m_{ii} \in [-1,0)$ and 
    $m_{ij} \geq 0$ for $j < i$, and let $C$ be a diagonal 
    positive semidefinite matrix. With $\Sigma$ the solution of 
    the Lyapunov equation, the following lower bound holds for the marginal 
    variance of node $d$:
    \begin{equation} \label{eq:lowerbound}
        \Sigma_{dd} \geq  - \frac{c_{dd}}{2 m_{dd}}  + 
        \frac{1}{2} M_{d1} \bSigma_{11} (M_{d1})^T.
    \end{equation}
\end{proposition}

The proof is given in Section \ref{sec:proof}. Note that \eqref{eq:lowerbound}
does not hold generally if the off-diagonal entries $m_{ij}$ are allowed to 
be negative. 

\begin{ex} \label{ex:simple}
    To illuminate what the lower bound in Proposition \ref{prop:lowerbound}
    says, and to illustrate the phenomenon that marginal variances tend to increase along 
    the topological order, we reconsider the simple acyclic model
    from Example \ref{ex:path} with $\zeta = \gamma = 1$. Then 
    \begin{equation} \label{eq:Sigmaij_simple}
        \Sigma_{ij} =  2^{-i - j - 1} \sum_{i_0=1}^{\min\{i,j\}} 4^{i_0}
        {i + j - 2i_0 \choose i - i_0}.
    \end{equation}
    
    From \eqref{eq:Sigmaij_simple} it follows directly that for $d \geq 2$,
    \begin{align}
            \Sigma_{dd} & = 2^{-2d - 1} \sum_{i_0=1}^{d} 4^{i_0}
        {2(d - i_0) \choose d - i_0} \label{eq:Sigmadd} \\
        & = 2^{-2(d-1)-1}\binom{2(d-1)}{d-1} + \underbrace{2^{-2d-1} \sum_{i_0=2}^d 4^{i_0} \binom{2(d-i_0)}{d-i_0}}_{ =  \Sigma_{(d-1)(d-1)}} > \Sigma_{(d-1)(d-1)}, \nonumber 
    \end{align}
    which shows that the marginal variance $\Sigma_{dd}$ is strictly increasing 
    as a function of $d$, see Figure \ref{fig:lowerbound}. This is not surprising, since the
    sink node $d$ accumulates the variances from all other nodes.
     
    To compare $\Sigma_{dd}$ to the lower bound in \eqref{eq:lowerbound}    
    we first derive a slightly different representation of $\Sigma_{dd}$. 
    Using that ${n \choose k} = \frac{n(n-1)}{k(n - k)} {n-2 \choose k-1}$ 
    for the second identity below, we get from \eqref{eq:Sigmadd} 
    that for $d \geq 2$,
    \begin{align*}
        \Sigma_{dd} 
        & =  \frac{1}{2} + 2^{-2d - 1}  \sum_{i_0=1}^{d-1} 4^{i_0}
        {2(d - i_0) \choose d - i_0} \\
        & = \frac{1}{2} + 2^{-2(d-1) - 1} 2^{-2} \sum_{i_0=1}^{d-1} 4^{i_0}
        \frac{2(d - i_0)(2(d - i_0) - 1)}{(d - i_0)^2} {2(d - 1 - i_0) \choose d - 1 - i_0} \\ 
        & = \frac{1}{2} + 2^{-2(d-1) - 1} \sum_{i_0=1}^{d-1} 4^{i_0}
        \left(1 - \frac{1}{2(d - i_0)}\right) {2(d - 1 - i_0) \choose d - 1 - i_0}. \\
    \end{align*}

\begin{figure}
    \centering
    \includegraphics[width = 0.9\textwidth]{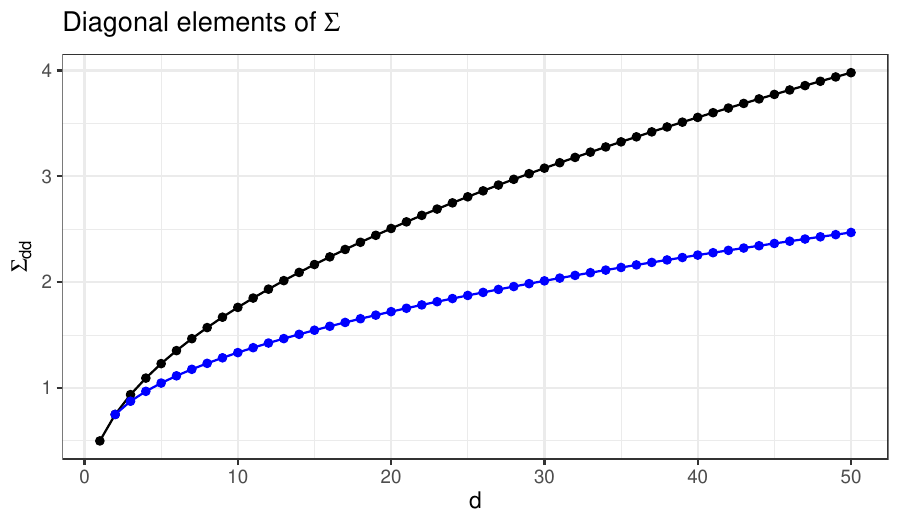}
    \caption{The variance $\Sigma_{dd}$ in Example \ref{ex:simple} 
    given by \eqref{eq:Sigmadd} as a function of $d$ (black) and the lower bound given by 
    \eqref{eq:lowerbound2} (blue). \label{fig:lowerbound}}
\end{figure}

    The second term above is $\Sigma_{(d-1)(d-1)}$ except from the factors 
    $(1 - 1/(2(d - i_0)))$ within the sum. Observe that 
    \[
        \frac{1}{2} \leq \left(1 - \frac{1}{2(d - i_0)}\right) < 1,
    \]
    with the factor being equal to $1/2$ for $i_0 = d-1$. This shows that
    \[
        \frac{1}{2} + \frac{1}{2} \Sigma_{(d-1)(d-1)} \leq \Sigma_{dd} < 
        \frac{1}{2} + \Sigma_{(d-1)(d-1)}.
    \]
    To compute the lower bound in \eqref{eq:lowerbound}, we first 
    note that $m_{dd} = -1$ and $c_{dd} = 1$, thus 
    \[
        - \frac{c_{dd}}{2 m_{dd}} = \frac{1}{2}.
    \]
    Moreover, $M_{d1} = (0, 0, \ldots, 1)$, thus $M_{d1} \bSigma_{11} (M_{d1})^T = 
    \Sigma_{(d-1)(d-1)}$. This shows that the lower bound in \eqref{eq:lowerbound}
    is indeed 
    \begin{equation}
        \label{eq:lowerbound2}
        \Sigma_{dd} \geq \frac{1}{2} + \frac{1}{2} \Sigma_{(d-1)(d-1)},
    \end{equation}
    as we also found above. The computations in this example show that 
    we cannot in general replace the factor $1/2$ in the bound by a larger 
    constant, and certainly not by $1$. The lower bound is also shown in Figure 
    \ref{fig:lowerbound}.
\end{ex}

\subsection{Proof of Proposition \ref{prop:lowerbound}} \label{sec:proof} 
Before giving the proof, we will need a few auxiliary results.

\begin{definition} For $a,b \in \mathbb{N}_0$ and $|z| < 1/2$  
    define
    \begin{equation}
            \label{eq:Hdef}
        H(a, b, z) = \sum_{n, m \in \mathbb{N}_0}
        \frac{(a + b + 1)_{n + m + 2}}{(a + 1)_{n + 1}(b + 1)_{m + 1}}
        z^{n + m}.
    \end{equation}
    Here $(a)_n = a(a+1)\cdots(a + n - 1)$ denotes the rising Pochhammer symbol.
\end{definition}

\begin{proposition} \label{prop:Drep}
    Let $\tau = d \leftarrow \tilde{\tau} \rightarrow d$ be a base trek from 
    $d$ to $d$. Then
    {\small
    \begin{equation*}
        D(\lambda_1, \ldots, \lambda_{d}, \tau)  
        = \sum_{\alpha, \beta \in \mathcal{N}(\tilde{\tau})} 
        H\left(n(\tilde{\tau}) + \alpha_{\bullet}, m(\tilde{\tau}) + \beta_{\bullet}, \frac{\lambda_d}{2}\right) 
        {l(\tilde{\tau}) + \alpha_{\bullet} + \beta_{\bullet}
        \choose n(\tilde{\tau}) + \alpha_{\bullet}}
        \prod_{i=1}^{d-1} \left(\frac{\lambda_i}{2}\right)^{\alpha_i + \beta_i}. 
    \end{equation*}
    }
\end{proposition}

\begin{proof} By Definition \ref{dfn:D} we have
\begin{align*}
    D(\lambda_1, \ldots, \lambda_{d}, \tau)  & = 
    \sum_{\alpha, \beta \in \mathcal{N}(\tau)}
    {l(\tau) + \alpha_{\bullet} + \beta_{\bullet} \choose n(\tau) + \alpha_{\bullet}}
    \prod_{i=1}^d \left(\frac{\lambda_i}{2}\right)^{\alpha_i + \beta_i} \\
    & = \sum_{\alpha, \beta \in \mathcal{N}(\tilde{\tau})} \sum_{\alpha_d, \beta_d \in \mathbb{N}_0}
    {l(\tilde{\tau}) + \alpha_{\bullet} + \beta_{\bullet} + \alpha_d + \beta_d + 2
    \choose n(\tilde{\tau}) + \alpha_{\bullet} + \alpha_d + 1}
    \prod_{i=1}^{d} \left(\frac{\lambda_i}{2}\right)^{\alpha_i + \beta_i}. 
\end{align*}
Recall that $l(\tilde{\tau}) = n(\tilde{\tau}) + m(\tilde{\tau})$, so with $a = n(\tilde{\tau}) + \alpha_\bullet$ 
and $b = m(\tilde{\tau}) + \beta_\bullet$ we have the following identity for the 
binomial coefficient 
\begin{align*}
    {l(\tilde{\tau}) + \alpha_{\bullet} + \beta_{\bullet} + \alpha_d + \beta_d + 2
    \choose n(\tilde{\tau}) + \alpha_{\bullet} + \alpha_d + 1} & = 
    {a + b + \alpha_d + \beta_d + 2 \choose a + \alpha_d + 1} \\
    & = \frac{(a + b + \alpha_d + \beta_d + 2)!}{(a + \alpha_d + 1)!(b + \beta_d + 1)!} \\
    & = \frac{(a + b + 1)_{\alpha_d + \beta_d + 2}}{(a + 1)_{\alpha_d + 1}(b + 1)_{\beta_d + 1}}
    \frac{(a + b)!}{a!b!} \\
    & = \frac{(a + b + 1)_{\alpha_d + \beta_d + 2}}{(a + 1)_{\alpha_d + 1}(b + 1)_{\beta_d + 1}}
    {l(\tilde{\tau}) + \alpha_{\bullet} + \beta_{\bullet}
    \choose n(\tilde{\tau}) + \alpha_{\bullet}}.
\end{align*}
Plugging this into the sum above and using the definition of $H$ in \eqref{eq:Hdef} gives
{\small 
\begin{align*}
    & D(\lambda_1, \ldots, \lambda_{d}, \tau)  \\ & = 
    \sum_{\alpha, \beta \in \mathcal{N}(\tilde{\tau})} \sum_{\alpha_d, \beta_d \in \mathbb{N}_0}
    \frac{(a + b + 1)_{\alpha_d + \beta_d + 2}}{(a + 1)_{\alpha_d + 1}(b + 1)_{\beta_d + 1}} \left(\frac{\lambda_d}{2}\right)^{\alpha_d + \beta_d}
    {l(\tilde{\tau}) + \alpha_{\bullet} + \beta_{\bullet}
    \choose n(\tilde{\tau}) + \alpha_{\bullet}} 
    \prod_{i=1}^{d-1} \left(\frac{\lambda_i}{2}\right)^{\alpha_i + \beta_i}
    \\
    & = 
    \sum_{\alpha, \beta \in \mathcal{N}(\tilde{\tau})} 
    H\left(n(\tilde{\tau}) + \alpha_{\bullet}, m(\tilde{\tau}) + \beta_{\bullet}, \frac{\lambda_d}{2}\right) 
    {l(\tilde{\tau}) + \alpha_{\bullet} + \beta_{\bullet}
    \choose n(\tilde{\tau}) + \alpha_{\bullet}}
    \prod_{i=1}^{d-1} \left(\frac{\lambda_i}{2}\right)^{\alpha_i + \beta_i}. 
\end{align*}
}
\end{proof}

\begin{lemma} \label{lem:2bound}
    For $a,b \in \mathbb{N}_0$ and $z \in [0, 1/2)$ we have
    \begin{equation}
        H(a, b, z) \geq 2
    \end{equation}    
\end{lemma}

\begin{proof} Note that $H(a, b, z) = H(b, a, z)$, so we may assume that $b \geq a$.
    Since $z \geq 0$ all terms in the sum defining $H(a, b, z)$ are non-negative and 
    we find that
    \begin{align*}
        H(a, b, z) & = \sum_{n, m \in \mathbb{N}_0}
        \frac{(a + b + 1)_{n + m + 2}}{(a + 1)_{n + 1}(b + 1)_{m + 1}}
        z^{n + m} \\
        & \geq \frac{(a + b + 1)_{2}}{(a + 1)_{1}(b + 1)_{1}} \\
        & = \frac{(a + b + 2)(a + b + 1)}{(a + 1)(b + 1)} \\
        & = \left(1 + \frac{b + 1}{a + 1}\right)
        \left( 1 + \frac{a}{b + 1} \right)  \geq 2.
    \end{align*}
\end{proof}

\begin{proof}[Proof of Proposition \ref{prop:lowerbound}] The base treks from $d$ to $d$ are either of the form 
    $d \unEdge d$ or $d \leftarrow \tilde{\tau} \rightarrow d$ for a 
    base trek $\tilde{\tau}$ from $i$ to $j$ with $i, j < d$.
    Since $\omega(M, C, d \unEdge d) = c_{dd}$, we have
     the following representation of $\Sigma_{dd}$:
    \begin{align*}
        & \Sigma_{dd} = \frac{c_{dd}}{2} \sum_{\alpha_d, \beta_d \in \mathbb{N}_0} 
        {\alpha_d + \beta_d \choose \alpha_d} \left(\frac{m_{dd} + 1}{2}\right)^{\alpha_d + \beta_d} \\
        &  + 
        \underbrace{\sum_{i = 1}^{d - 1} \sum_{j= 1}^{d-1} \sum_{\tilde{\tau} \in \mathcal{T}_0(i, j)} 
        2^{-l(\tilde{\tau}) - 2 - 1} D(m_{11} + 1, \ldots, m_{dd} + 1, d \leftarrow \tilde{\tau} \rightarrow d)
        \omega(M, C, d \leftarrow \tilde{\tau} \rightarrow d)}_{= R}.
    \end{align*}

    Using the negative binomial and geometric series, the sum in the first term equals
    \begin{align*}
        \sum_{\alpha_d, \beta_d \in \mathbb{N}_0} 
        {\alpha_d + \beta_d \choose \alpha_d} \left(\frac{m_{dd} + 1}{2}\right)^{\alpha_d + \beta_d}
        & = \sum_{\beta_d \in \mathbb{N}_0} \left(\frac{m_{dd} + 1}{2}\right)^{\beta_d}   
        \sum_{\alpha_d \in \mathbb{N}_0}   {\alpha_d + \beta_d \choose \alpha_d} \left(\frac{m_{dd} + 1}{2}\right)^{\alpha_d} \\
        & = \sum_{\beta_d \in \mathbb{N}_0} \left(\frac{m_{dd} + 1}{2}\right)^{\beta_d}
        \left(1 - \frac{m_{dd} + 1}{2}\right)^{- \beta_d - 1}
        \\
        & = \frac{2}{1 - m_{dd}} \sum_{\beta_d \in \mathbb{N}_0} \left(\frac{1 + m_{dd}}{1 - m_{dd}}\right)^{\beta_d}
        \\
        & =  \frac{2}{(1 - m_{dd})} \frac{1}{\left(1 - \frac{1 + m_{dd}}{1 - m_{dd}}\right)} \\
        & = - \frac{1}{m_{dd}}. 
    \end{align*}

    For the second term $R$, we first note that for $\tilde{\tau} \in \mathcal{T}_0(i, j)$,
    \begin{align*}
        \omega(M, C, d \leftarrow \tilde{\tau} \rightarrow d) & = 
        m_{di} m_{dj} \omega(M, C, \tilde{\tau}). 
    \end{align*}

    By Proposition \ref{prop:Drep} and Lemma \ref{lem:2bound} we have 
    \begin{align*}
        D(m_{11} + 1, \ldots, m_{dd} + 1, d \leftarrow \tilde{\tau} \rightarrow d) & \geq 
        2 \sum_{\alpha, \beta \in \mathcal{N}(\tilde{\tau})} 
        {l(\tilde{\tau}) + \alpha_{\bullet} + \beta_{\bullet}
        \choose n(\tilde{\tau}) + \alpha_{\bullet}}
        \prod_{k=1}^{d-1} \left(\frac{m_{kk} + 1}{2}\right)^{\alpha_k + \beta_k} \\
        & = 2  D(m_{11} + 1, \ldots, m_{dd} + 1, \tilde{\tau}).
    \end{align*}
    As $m_{ij} \geq 0$ for $j < i$, $\omega(M, C, \tilde{\tau})  \geq 0$ for a base trek $\tilde{\tau}$, 
    and using the above inequality within the sum in the second term $R$ gives
    \begin{align*}
        R & \geq \frac{1}{2} \sum_{i = 1}^{d - 1}  \sum_{j= 1}^{d-1}  m_{di} m_{dj}
        \sum_{\tilde{\tau} \in \mathcal{T}_0(i, j)}
        2^{-l(\tilde{\tau}) - 1} D(m_{11} + 1, \ldots, m_{dd} + 1, \tilde{\tau}) 
        \omega(M, C, \tilde{\tau}) \\
        & = \frac{1}{2} \sum_{i = 1}^{d - 1}  \sum_{j= 1}^{d-1}  m_{di} m_{dj}
        \Sigma_{ij} = \frac{1}{2} M_{d1} \bSigma_{11} (M_{d1})^T.
    \end{align*}
    Combining these results gives 
    \begin{align*}
        \Sigma_{dd} & = - \frac{c_{dd}}{2m_{dd}} + R
        \geq  - \frac{c_{dd}}{2 m_{dd}}  + 
        \frac{1}{2} M_{d1} \bSigma_{11} (M_{d1})^T.
    \end{align*}
\end{proof}

\begin{rem} The proof of Proposition \ref{prop:lowerbound} shows that we can always write 
    \[
        \Sigma_{dd} = - \frac{c_{dd}}{2 m_{dd}} + R,
    \]
    which is directly comparable to \eqref{eq:LAN2} for the linear additive noise model.
    The general expression for $R$ is somewhat complicated, but the proof above 
    shows that it can be lower bounded by a simpler expression whenever $m_{ij} \geq 0$.
    It is an open question if other assumptions lead to either bounds or simplifications of $R$. 
\end{rem}

\section{Concluding remarks}

Trek rules are useful for linking the entries of the solution to the Lyapunov
equation to graphical properties of the underlying mixed graph. A
straightforward observation is that $\Sigma_{ij} = 0$ if there is no trek from
$i$ to $j$. The general trek rule as stated in Proposition \ref{prop:trek} 
and Theorem \ref{thm:altrep} is, however,
more complicated than the well known trek rule for the linear additive noise
model, and even in the acyclic case it does not simplify completely to a
polynomial representation in general. This is due to the self-loops. 
For acyclic models the trek rule in Theorem \ref{thm:altrep} is, 
nevertheless, a polynomial in the off-diagonal entries of $M$ and the entries of $C$. 

The trek rule is not generally useful for the numerical computation of the 
solution to the Lyapunov equation, but it can be used in special cases to derive 
either explicit formulas for solutions or to derive bounds on entries of 
the solution. We have illustrated this for the acyclic 
models in Examples \ref{ex:simple} and \ref{ex:factor}, where 
much simpler polynomial trek rules are possible. As an example of a
non-trivial application of the trek rule, we derived the lower bound in
Proposition \ref{prop:lowerbound} on the marginal variance $\Sigma_{dd}$ for a
stable acyclic model with a diagonal $C$ matrix and off-diagonal drift
entries being non-negative. 

\section{Acknowledgements}

The trek rules in Proposition \ref{prop:trek} and Theorem \ref{thm:altrep}
were discovered while the author was visiting Mathias Drton at the Technical University of
Munich in November and December 2021. I am grateful to Mathias for hosting me.
More recently, Alexander Reisach mentioned to the author that lower bounds on 
the variances for acyclic models would be of interest, and the trek rule seemed 
well suited for this purpose. I am grateful to Alexander for this suggestion.
I also thank Mathias Drton, Sarah Lumpp and the three anonymous reviewers 
for a number of suggestions that improved the manuscript.
The author was supported by a research grant (NNF20OC0062897) from Novo Nordisk Fonden.

\bibliographystyle{agsm}
\bibliography{Trek_alg_stat.bib}

\end{document}